\documentclass[12pt]{amsart}
\usepackage{amssymb,amsthm,amsmath,pb-diagram,color,enumerate}
\usepackage{mathrsfs}
\usepackage[all]{xy}
\usepackage{url}
\usepackage{fullpage}
\usepackage[OT2,T1]{fontenc}
\DeclareSymbolFont{cyrletters}{OT2}{wncyr}{m}{n}
\DeclareMathSymbol{\Sha}{\mathalpha}{cyrletters}{"58}

\newtheorem{thm}[equation]{Theorem}
\newtheorem{lemma}[equation]{Lemma}
\newtheorem{lem}[equation]{Lemma}
\newtheorem{prop}[equation]{Proposition}
\newtheorem{cor}[equation]{Corollary}
\theoremstyle{definition}
\newtheorem{remark}[equation]{Remark}

\newtheorem{hyp}[equation]{Hypothesis}
\newtheorem{notation}[equation]{Notation}
\newtheorem{dfn}[equation]{Definition}

\newtheorem{example}[equation]{Example}

\newcommand{\mbb}[1]{\mathbb #1}

\newcommand{\mc}[1]{\mathcal #1}

\newcommand{\oper}[1]{\operatorname{#1}}

\newcommand{\XX}{\mc X}
\newcommand{\YY}{\mc Y}
\newcommand{\ZZ}{\mc Z}
\newcommand{\wh}{\widehat}
\newcommand{\til}{\widetilde}

\newcommand{\Spec}{\oper{Spec}}

\newcommand{\Gal}{\oper{Gal}}

\newcommand{\cha}{\oper{char}}

\newcommand{\Frac}{\oper{frac}}

\newcommand{\h}{^{\oper{h}}}
\newcommand{\cd}{\oper{cd}}
\newcommand{\tor}{\oper{tor}}
\def\<{\left<}
\def\>{\right>}

\newcounter{itemcounter}

\numberwithin{equation}{section}

\title{Local-global principles for zero-cycles on homogeneous spaces over arithmetic function fields}
\author{J.-L. Colliot-Th\'el\`ene, D. Harbater, J. Hartmann, \\D. Krashen, R. Parimala, and V. Suresh}
\date{\today}

\begin{document}

\thanks{
\textit{Mathematics Subject Classification} (2010): Primary 14C25, 14G05, 14H25; Secondary 11E72,
12G05, 12F10.
\textit{Key words and phrases.} Linear algebraic groups and torsors, zero-cycles, local-global principles, semi-global fields, discrete valuation rings.}

\maketitle

\begin{abstract}
We study the existence of zero-cycles of degree one on 
varieties that are defined over a function field of a curve over a 
complete discretely valued field. We show that 
local-global principles hold for such zero-cycles provided that 
local-global principles hold for the existence of rational points over 
extensions of the function field.  This assertion is analogous to a 
known result concerning varieties over number fields.  Many of our results are shown to hold more generally in the henselian case.
\end{abstract}

\section{Introduction}

The study of rational points on varieties is a fundamental subject of arithmetic geometry. Local-global principles (and their obstructions) are one of the main tools in understanding whether a rational point exists. A related object of study is the index of an algebraic variety, and one may ask whether it is equal to one, i.e., whether the variety admits a zero-cycle of degree one. Equivalently,  for every prime $\ell$, does there exist a point defined over a finite field extension of degree prime to $\ell$? 

In this paper, we consider varieties over certain two-dimensional fields; in particular, one-variable function fields over complete discretely valued fields (so-called {\em semi-global} fields). These fields are amenable to patching methods. Local-global principles for rational points on homogeneous spaces over such fields were studied for example in \cite{CTPS12}, \cite{CTPS16}, \cite{HHK:H1}. Here we exhibit several situations where such local-global principles imply corresponding local-global statements for zero-cycles. Parallel results over number fields were first obtained by Liang (Prop.~3.2.3 in~\cite{liang}; see also \cite{CTAWS}, Section~8.2). However, our situation involves substantial new difficulties to overcome.

A semi-global field admits several natural collections of overfields with respect to which local-global principles can be studied. After a choice of normal projective model of the semi-global field, one may consider two distinct
collections of overfields: the first associated with {\em patches} on such a model (described in the beginning of Section~\ref{LGP sec}), and the second consisting of overfields which are fraction fields of complete local rings at points of the closed fiber of the model. Finally, as in the number field case, one may also work with the set of completions with respect to discrete valuations.

Consider a scheme~$Z$ of finite type over a semi-global field~$F$. We show (Theorem~\ref{sep LGP prime patches}) that in the case of overfields coming from patching, a local-global principle for rational points for the base change $Z_E$ to all finite separable field extensions $E/F$ implies a local-global principle for separable zero-cycles of degree one over $F$ (or analogously, of degree prime to some given prime $\ell$). This also gives a local-global principle for the (separable) index of $Z$ (Corollary~\ref{sep LGP patches}). The analogous results hold when the collection of overfields under consideration comes from points of the closed fiber of a model as described above (Theorem~\ref{LGP closed fiber}). Moreover, these latter results extend to the case of function fields of curves over excellent henselian discrete valuation rings. In particular, we obtain local-global principles for zero-cycles in that situation; see Proposition~\ref{LGP cl fiber henselian}. 
The required local-global hypothesis for rational points over a semi-global field holds, for instance, if $Z$ is a torsor under a linear algebraic groups that is connected and rational (see Corollary~\ref{LGP ex}).

In the situation where $Z$ is a principal or projective homogeneous space under a linear algebraic group, we also obtain local-global principles with respect to discrete valuations (Theorem~\ref{LGPdvr}) under additional hypotheses on $F$ (e.g., when $F$ is the function field of a curve over a complete discretely valued field with algebraically closed residue field of characteristic zero). In certain cases, the existence of local zero-cycles of degree one already implies the existence of a global rational point (Theorem~\ref{cd2tors}, Theorem~\ref{LGPproj}). 

The results are obtained using a combination of methods. Many of the local-global statements rely on descent results that we prove for finite field extensions and for the existence of rational points,
in context of a pair of fields $L \subseteq L'$ and a finite separable field extension $E'/L'$.
In the former type of descent result (e.g., Proposition~\ref{branch to point}), 
we find a finite separable field extension $E/L$ such that $E \otimes_L L' \cong E'$.  In the latter type (e.g.,  Proposition~\ref{globalize_FP_extension}), given an $L$-scheme $Z$ with an $E'$-point, we find an appropriate finite separable field extension $E/L$ such that $Z$ has a point over $E \otimes_L L'$.  Descent of extensions of fields arising in the context of patching is also studied in \cite{CTHHKPS:Kampen}, which builds on the results here.  For our local-global principles with respect to discrete valuations, we also use structural properties of linear algebraic groups, as well as results about nonabelian cohomology in degree two.

The manuscript is organized as follows. Section~\ref{LGP sec} contains descent results. The first two subsections concern fields occurring in patching. We show that finite separable field extensions of some of these overfields descend to the semi-global field $F$ (Subsection~\ref{extension descent subsec}); in other cases it is still possible to descend the existence of points (Subsection~\ref{points descent subsec}; in particular, see Prop.~\ref{globalize_FP_extension}). Subsection~\ref{local descent subsec} contains local descent results.
Local-global principles for zero-cycles are proven in Section~\ref{lgp sec}; this is first done with respect to patches and points  on arithmetic curves over complete discrete valuation rings (Subsection~\ref{lgp patches points subsec}), and then generalized to excellent henselian valuation rings (Subsection~\ref{lgp henselian subsec}). In certain cases we obtain local-global principles with respect to discrete valuations (Subsections~\ref{lgp dvr subsec} and~\ref{lgp dvr alg cl subsec}), for principal and projective homogeneous spaces over certain 2-dimensional fields, including semi-global fields.

Research on this subject was partially carried out during a visit of the authors at the American Institute of Mathematics (AIM). We thank AIM for the productive atmosphere and wonderful hospitality.

\section{Descent  results}\label{LGP sec}

This section contains descent results which will be essential in proving the local-global principles in the following section. We first consider collections of fields coming from patching.
We recall the following notation, which was established in \cite{HH:FP}, \cite{HHK}, and \cite{HHK:refinements}: 

\begin{dfn} \label{basic def}
Let $K$ be a discretely valued field with valuation ring $T$, uniformizer~$t$, and residue field $k$. Let $F$ be a one-variable function field over $K$; i.e., a finitely generated field extension of transcendence degree one in which $K$ is algebraically closed.
A {\em normal model} of $F$ is an integral $T$-scheme $\XX$ with function field $F$ that is flat and projective over $T$ of relative dimension one, and that is normal as a scheme. If in addition the scheme $\XX$ is regular, we say that $\XX$ is a {\em regular model}. The {\em closed fiber} of $\mc X$ is $\mc X_k := \mc X \times_T k$.  

In the case that the discrete valuation ring $T$ is complete, we call $F$ a {\em semi-global field}.  (We will also often consider the more general case that $T$ is excellent and henselian.) 
\end{dfn}

The following notation will be used throughout this manuscript.

\begin{notation} \label{basic notation}
In the context of Definition~\ref{basic def}, let $\XX$ be a normal model for $F$. 
If $P$ is a (not necessarily closed) point of the closed fiber $X$ of $\XX$, let $R_P$ be the local ring of $\XX$ at $P$; let $\wh R_P$ be its completion with respect to the maximal ideal $\frak m_P$; and let $F_P$ be the fraction field of $\wh R_P$.  In the case that $P$ is a closed point of $X$, the {\em branches} of $X$ at $P$ are the height one prime ideals of $\wh R_P$ that contain $t$.  We write $R_\wp$ for the local ring of $\wh R_P$ at a branch $\wp$.  This  is a discrete valuation ring.  We write $\wh R_\wp$ for its completion, and $F_\wp$ for the fraction field of $\wh R_\wp$.

If $U$ is a non-empty connected affine open subset of $X$, 
then we write $R_U$ for the subring of $F$ consisting of rational functions that are regular at each point of $U$.  We let $\wh R_U$ be the $t$-adic completion of $R_U$. 
This is an integral domain by \cite[Proposition~3.4]{HHK:refinements}, and we let $F_U$ be the fraction field of $\wh R_U$.  
If $P \in U \subseteq U'$, then $\wh R_{U'} \subseteq \wh R_U \subset \wh R_P$ and $F_{U'} \subseteq F_U \subset F_P$.
\end{notation}

The spectra of the rings $\wh R_P$ and $\wh R_U$ above are thought of as ``patches'' on $\XX$.  For the sake of readers who are familiar with rigid geometry, we remark that our fields $F_P$ and $F_U$ are the same as the rings of meromorphic functions on the corresponding rigid open and rigid closed affinoid sets obtained by deleting the closed fiber from the patches; see \cite{Ray}, concerning the relationship between formal schemes and rigid analytic spaces.

\subsection{Descent of field extensions}\label{extension descent subsec}

The following two statements generalize Proposition 3.5 of \cite{HHK:refinements}; by the {\em trivial \'etale algebra (of degree $n$)} over a field $L$ we mean the direct product of $n$ copies of $L$.

\begin{prop}\label{branch to point}
Let $\XX$ be a normal model of a semi-global field $F$, 
let $P$ be a closed point of $\XX$, let $\wp$ be a branch of the closed fiber $X$ at~$P$, and
let $E_\wp$ be a finite separable field extension of $F_\wp$.  Then 
there exists a finite separable field extension $E_P$ of $F_P$ such that 
$E_P \otimes_{F_P} F_\wp\cong E_\wp$ as extensions of~$F_P$, and such that 
$E_P$ induces the trivial \'etale algebra over $F_{\wp'}$ for every other branch
$\wp'$ at $P$.
\end{prop}

\begin{proof}  
Let $f_\wp(x) \in F_\wp[x]$ be the minimal polynomial of a primitive element of $E_\wp$.
For each other branch $\wp'$ at $P$, let $f_{\wp'}(x)\in F_{\wp'}[x]$ be a separable polynomial of the same degree that
splits completely over $F_{\wp'}$ and thus defines the trivial \'etale algebra over $F_{\wp'}$.  
The field $F_P$ is dense in $\prod F_{\wp'}$ by Theorem VI.7.2.1 of \cite{Bourbaki},
where the product ranges over all the branches 
at $P$ (including $\wp$).  So applying Krasner's Lemma (e.g., \cite[Prop.~II.2.4]{Lang}) to the above polynomials yields the 
desired extension of $F_P$, which is a field since $E_\wp$ is.
\end{proof}

\begin{prop} \label{component to global}
Let $\XX$ be a normal model of a semi-global field $F$, let $U$ be a non-empty connected affine open subset of the closed fiber of $\XX$, and let $E_U$ be a finite separable field extension of $F_U$.  Then there is a finite separable field extension $E$ of $F$ such that $E \otimes_F F_U \cong E_U$ as extensions of~$F_U$.
\end{prop}

\begin{proof}
Let $\bar U$ be the closure of $U$ in the closed fiber $X$ of $\mc X$.  
After blowing up $\XX$ at the points of $\mc P_U:= \bar U \smallsetminus U$ (which changes the model $\XX$ but does not change $F$, $U$, or $F_U$), we may assume that $\bar U$ is unibranched at each point of $\mc P_U$ (see \cite{Lip}, Lecture~1).  Let $\mc P$ be a finite set of closed points of $X$ that satisfies $\mc P \cap \bar U = \mc P_U$, and which contains at least one point on each irreducible component of $X$.  
Let $\mc U$ be the set of connected components of $X \smallsetminus \mc P$.  Then $U\in \mc U$, and each element of $\mc U$ is affine.

For each point $P \in \mc P_U$, consider the unique branch $\wp$ at $P$ on $U$.
Then $E_\wp := E_U \otimes_{F_U} F_\wp$ 
is a finite direct product of finite separable field extensions $E_{\wp, i}$ of $F_\wp$.  
By Proposition~\ref{branch to point}, for each $i$ there is a finite separable field extension $E_{P,i}$ of $F_P$ 
such that $E_{P,i} \otimes_{F_P} F_\wp \cong E_{\wp, i}$ and such that $E_{P,i} \otimes_{F_P} F_{\wp'}$ is a 
trivial \'etale algebra over $F_{\wp'}$ for every other branch $\wp'$ at $P$.
So the direct product of the fields $E_{P,i}$ (ranging over $i$) is a finite \'etale $F_P$-algebra $E_P$ 
satisfying $E_P \otimes_{F_P} F_\wp \cong E_\wp$ and such that 
$E_P \otimes_{F_P} F_{\wp'}$ is the trivial \'etale algebra of degree $n := [E_U:F_U]$ over $F_{\wp'}$.  Here $E_P$ is well 
defined for each $P \in \mc P_U$, since $\wp$ is unique given $P$.  

For every $P \in \mc P$ that is not in $\mc P_U$, let $E_P$ be the trivial \'etale algebra of degree $n$ over 
$F_P$.  Similarly, for every $U' \in \mc U$ other than $U$ let $E_{U'}$ be the trivial \'etale algebra of 
degree $n$ over $F_{U'}$, and for every branch $\wp $ at a point that is not in $\bar U$ let $E_\wp$   
be the trivial \'etale algebra of degree $n$ over $F_\wp$.  Thus for {\em every} branch $\wp $ at a point $P \in \mc P$ lying on some $U' \in \mc U$ (including the case $U'=U$), we have isomorphisms $E_P \otimes_{F_P} F_\wp \cong E_\wp \cong E_{U'} \otimes_{F_{U'}} F_\wp$.  But patching holds for finite separable algebras in this context; see Proposition~3.7 and Example 2.7 in \cite{HHK:refinements}. 
So there is a finite \'etale $F$-algebra $E$ that compatibly induces all the algebras $E_P, E_{U'}, E_\wp$.  Since $E_U$ is a field, so is $E$. 
\end{proof}

The above two propositions suggest analogous statements in which the roles of $P$ and $U$ are interchanged:  The analog of Proposition~\ref{branch to point} would assert that if $\wp$ is a branch at a point $P$ in $\bar U \smallsetminus U$, then every finite separable field extension of $F_\wp$ would be induced by a finite separable field extension of $F_U$.  The analog of Proposition~\ref{component to global} would say that every finite separable field extension of $F_P$ is induced by a finite separable field extension of $F$.  This does not hold in general, as the following example shows.

\begin{example}
Let $T$ denote the complete discrete valuation ring $k[[t]]$, where $k$ is a field of characteristic $p>0$.
Let $\XX$ be the projective $x$-line over $T$, and let $P$ be the origin on the projective $k$-line $X$. Then $F_P$ equals $k((t,x))$, the fraction field of $k[[t,x]]$.  Consider the field extension $E_P/F_P$ generated by the solutions of $y^p-y=\frac{\alpha}{t}$, where $\alpha$ is a transcendental power series (i.e., a power series in $x$ transcendental over $F=k((t))(x)$). Then one can show that $E_P$ is not induced by an extension of $F$ in the above sense.
In fact, this example is a special instance of \cite{CTHHKPS:Kampen}, Lemma 2.14, to which we refer the reader for a proof.

\end{example}

In \cite[Section~2]{CTHHKPS:Kampen}, it is shown that versions of descent for field extensions with the roles of $P$ and $U$ interchanged do hold when the residue field $k$ of $K$ has characteristic zero, and that as a consequence there are local-global principles in that situation.  To treat the more general case, 
we consider a different type of descent in the next subsection.

\subsection{Descent of existence of  points}\label{points descent subsec}

In order to prove a local-global principle (with respect to patches) for zero-cycles in arbitrary characteristic, we prove a descent result for the existence of points instead of field extensions (Proposition~\ref{globalize_FP_extension} below).  Specifically, let 
$T$ be an excellent henselian (e.g., complete) discrete valuation ring, and choose a normal model $\XX$ of a one-variable function field $F$ over the fraction field $K$ of $T$.  In the context of
Notation~\ref{basic notation}, we show that if 
$E_P/F_P$ is a separable field extension whose degree is not divisible by some prime number $\ell$, 
and if $Z$ is an $F$-scheme of finite type
which has an $E_P$-point, then there is a finite separable field extension $E/F$ of degree prime to $\ell$ such that $Z$ has a point over $E \otimes_F F_P$.  
First, some preparation is needed.

For $P$ a point on the closed fiber of $\XX$ as above, the henselization $R_P\h$ of $R_P$
is the same as the henselization at $P$ of the coordinate ring of an affine open subset of $\XX$ that contains $P$.  Since that coordinate ring is of finite type over the excellent henselian discrete valuation ring~$T$, Artin's Approximation Theorem (Theorem~1.10 of \cite{artin}) applies to a system of polynomial equations over $R_P\h$, and asserts that if there is a solution over $\wh R_P$ then there is a solution over $R_P\h$.  By clearing denominators, the same assertion holds with $\wh R_P$ and $R_P\h$ replaced by $F_P$ and $F_P\h$, where $F_P\h$ is the fraction field of~$R_P\h$ (this being a separable field extension of $F$).  We use this in the proof of the next proposition.

Note that we may pick a fixed algebraic closure $\bar F_P$ of $F_P$, and let $\bar F$ be the algebraic closure of $F$ in $\bar F_P$.  Thus $\bar F$ is an algebraic closure of $F$ that contains $F_P\h$.

\begin{prop} \label{henselization}
Let $T$ be an excellent henselian discrete valuation ring with fraction field $K$.  Let $F$ be a one-variable function field over $K$ (e.g., a semi-global field, if $K$ is complete), and let $\XX$ be a normal model
of $F$.  Let $P$ be a (not necessarily closed) point of the closed fiber of $\XX$, and let $Z$ be an $F$-scheme of finite type.
Let $\ell$ be a prime number, and suppose that there is a finite separable field extension $E_P/F_P$ of degree prime to $\ell$ such that $Z(E_P)$ is non-empty.  Then there is a finite separable field extension $E_P'/F_P\h$ of degree prime to $\ell$ such that $Z(E_P')$ is non-empty.
\end{prop}

\begin{proof}
Let $d$ be the smallest positive integer that is prime to $\ell$ such that there is a finite separable field extension $E_P/F_P$ of degree $d$ for which $Z$ has an $E_P$-point $\xi$.  Since $E_P$ is separable over $F_P$, there are exactly $d$ distinct $F_P$-embeddings $\sigma_1,\dots,\sigma_d$ of $E_P$ into an algebraic closure $\bar F_P$ of $F_P$.  By the minimality of $d$ and $E_P$, the point $(\sigma_1(\xi),\dots,\sigma_d(\xi)) \in Z^d(\bar F_P)$ does not lie on the closed subset $\Delta \subset Z^d$ where two or more of the entries are equal.  Consider the image $\mc D \subset S^d(Z)$ of $\Delta$ in the $d$-th symmetric power of $Z$; i.e., 
$S^d(Z) \smallsetminus \mc D = (Z^d \smallsetminus \Delta)/S_d$.  The image $\zeta \in S^d(Z) \smallsetminus \mc D$ of $(\sigma_1(\xi),\dots,\sigma_d(\xi))$ is an $F_P$-point on this $F$-scheme, corresponding to a morphism $\Spec(F_P) \to S^d(Z) \smallsetminus \mc D$.  The image of this morphism is a point of (the underlying topological space of) $S^d(Z) \smallsetminus \mc D$, and this lies in some affine open subset 
$\Spec(A) \subseteq S^d(Z) \smallsetminus \mc D$.  This point corresponds to a solution over $F_P$ to a system of polynomial equations over $F$ that defines $A$.  

By Artin's Approximation Theorem, there is a solution to this system of equations over the field $F_P\h$.  This corresponds to an $F_P\h$-point $\zeta'$ on $S^d(Z) \smallsetminus \mc D$.  Pick a point on $Z^d \smallsetminus \Delta$ that maps to $\zeta'$.  Each entry lies in an algebraic closure of $F_P\h$, or equivalently of $F$.  The $d$ entries are distinct, and this set of entries is stable under the absolute Galois group of $F_P\h$, since $\zeta'$ is defined over $F_P\h$.  Thus the entries form a disjoint union of orbits under this absolute Galois group, say of orders $d_1,\dots,d_r$, with $\sum_i d_i = d$.  Since $d$ is prime to $\ell$, so is some $d_i$.  Let $\xi'$ be an entry lying in the $i$-th orbit; this defines a point of $Z$, say with 
field of definition $E_P'$.  Then the field $E_P'$ is separable over $F_P\h$ by the distinctness of the entries; the degree of $E_P$ over $F_P\h$ is $d_i$, which is prime to $\ell$; and $\xi'$ is an $E_P'$-point of $Z$.
\end{proof} 

The above proof actually shows more, viz.~that if $[E_P:F_P]$ is minimal 
for the given property, then $E_P'$ can be chosen so that $[E_P':F_P\h] 
= [E_P:F_P]$.  This follows from the fact that at the end of the proof, 
$d_i$ must equal $d$ (i.e., there is just one orbit) by minimality of 
$d$ and because $\xi'$ induces a point of $Z$ over an extension of $F_P$ 
of degree at most $d_i$.

\begin{lemma} \label{finite_exten_la}
Let $L \subseteq L' \subseteq E'$ be separable algebraic field extensions, where $[E':L']$ is finite.  Let $Z$ be an $L$-scheme of finite type such that
$Z(E')$ is non-empty.  Then there 
are finite separable field extensions $L \subseteq \til L \subseteq \til E$ 
such that $\til L \subseteq L'$ and $\til E \subseteq E'$; 
$[\til E:\til L]=[E':L']$; $Z(\til E)$ is non-empty; 
and $E'$ is the compositum of its subfields $\til E$ and $L'$.
\end{lemma}

\begin{proof}
Let $\xi \in Z(E')$.  Since $Z$ is an $L$-scheme of finite type, there is an affine open subset of $Z$ that contains $\xi$ and is $L$-isomorphic to a Zariski closed subset $Y$ of $\mbb A^n_L$ for some~$n$.  Let $y_1,\dots,y_n \in E'$ be the coordinates of the image of $\xi$ in $Y$.  Let $z$ be a primitive element of the finite separable field extension $E'/L'$, say with minimal monic polynomial $g$ over $L'$ of degree $d = [E':L']$.  Thus each $y_i$ is of the form 
$\sum_{j=0}^{d-1}c_{ij}z^j$ with $c_{ij} \in L'$.  Let $\til L$ be the subfield of $L'$ generated over $L$ by the coefficients of $g$ and by the elements $c_{ij}$; this is finite over $L$, and it is separable over $L$ since $L'$ is.  The polynomial $g$ is irreducible over $\til L$ because it is irreducible over $L'$.  Thus $\til E := \til L(z) \subseteq E'$ is separable and of degree $d$ over $\til L$, and $\til E \subseteq \til E L' = L'(z) = E'$.  Also, each $y_i$ lies in $\til E$, since $z \in \til E$ and $c_{ij} \in \til L \subseteq \til E$.  So $(y_1,\dots,y_r) \in Y(\til E)$ and thus $\xi \in Z(\til E$). 
\end{proof}

\begin{lem} \label{group_lemma}
Let $\ell$ be a prime number, and let $L \subseteq \til L \subseteq \til E$ be finite separable field extensions such that $[\til E:\til L]$ is prime to $\ell$.  Let $\wh E$ be the Galois closure of $\til E/L$.  Then for every Sylow $\ell$-subgroup $S$ of $\Gal(\wh E/L)$, there is some $\sigma \in \Gal(\wh E/L)$ such that the compositum $\til L \wh E^S \subseteq \wh E$ contains $\sigma(\til E)$.
\end{lem}

\begin{proof}
The intersection $\Gal(\wh E/\til L)\cap S \subseteq \Gal(\wh E/L)$ is an $\ell$-subgroup of $\Gal(\wh E/\til L)$, and so it is contained in a Sylow $\ell$-subgroup $S^*$ of $\Gal(\wh E/\til L)$.  Let $S'$ be a Sylow $\ell$-subgroup of $\Gal(\wh E/\til E)$.  Since $[\til E:\til L]$ is prime to $\ell$, the group $S'$ is also a 
Sylow $\ell$-subgroup of $\Gal(\wh E/\til L)$.  Thus $S^*,S'$ are conjugate subgroups of $\Gal(\wh E/\til L)$, say by an element $\sigma \in \Gal(\wh E/\til L) \subseteq \Gal(\wh E/L)$.  Since $S' \subseteq \Gal(\wh E/\til E)$, its conjugate $S^* = (S')^\sigma$ is contained in $\Gal(\wh E/\sigma(\til E))$.  Thus $\wh E^{S^*}$ contains $\sigma(\til E)$.  So $\til L \wh E^S = \wh E^{\Gal(\wh E/\til L)} \wh E^S = \wh E^{\Gal(\wh E/\til L) \cap S}$, which contains $\wh E^{S^*}$ and hence contains $\sigma(\til E)$.
\end{proof}

Recall that if $L$ is a field, $Z$ is an $L$-scheme of finite type, and $A/L$ is a finite direct product of field extensions $L_i/L$, an $A$-point on $Z$ is a collection of points in $Z(L_i)$ for each~$i$. In particular, $Z(A)$ is nonempty if and only if $Z(L_i)$ is nonempty for all $i$.

\begin{prop} \label{abstract globalize prop}
Let $\ell$ be a prime number, and let $L \subseteq L' \subseteq E'$ be separable algebraic field extensions, where $[E':L']$ is finite and prime to $\ell$.  Let $Z$ be an $L$-scheme of finite type such that
$Z(E')$ is non-empty. 
Then there is a finite separable field extension $E/L$ of degree prime to $\ell$ such that $Z(E \otimes_L L')$ is non-empty. 

\end{prop}

\begin{proof} 
Let $\til L$ and $\til E$ be as in Lemma~\ref{finite_exten_la}; in particular, $[\til E:\til L]=[E':L']$ is prime to $\ell$.  Let $\wh E$ be the Galois closure of $\til E$ over $L$; let $S$ be a Sylow $\ell$-subgroup of $\Gal(\wh E/L)$; and let $E$ be the fixed field $\wh E^S$.  We will show that $E$ has the desired properties.

Since $S$ is a Sylow $\ell$-subgroup of $\Gal(\wh E/L)$, the degree of $E$ over $L$ is prime to $\ell$.  
In order to show that $Z(E \otimes_L L')$ is non-empty, it is sufficient to show that $Z(E \otimes_L \til L)$ is non-empty, because $L \subseteq \til L \subseteq L'$.  
Since $E/L$ is a separable field extension of finite degree, 
$E \otimes_L \til L$ is a finite separable algebra over $\til L$, and hence is a finite direct product $\prod \til E_i$ of finite separable field extensions of $\til L$.  It therefore suffices to show
that $Z(\til E_i)$ is non-empty for all~$i$.  Here each $\til E_i$ is isomorphic to a compositum of $\til L$ and $E$ with respect to some $L$-embeddings of those two fields into the common overfield $\wh E$ (since $\wh E/L$ is Galois).  

After conjugating, we may assume that the above $L$-embedding of $\til L$ into $\wh E$ is the given one (i.e., the original composition $\til L \hookrightarrow \til E \hookrightarrow \wh E$), while allowing the $L$-embedding of $E$ into $\wh E$ to vary.  The images of $E$ in $\wh E$ under the various $L$-algebra embeddings are just the Galois conjugates $\sigma(E)$ for $\sigma \in \Gal(\wh E/L)$.  Since $E$ is the fixed field $\wh E^S$, its Galois conjugates are the fixed fields of the conjugates of $S$, viz.\ the fields $\wh E^{S'}$ where $S'$ varies over the Sylow $\ell$-subgroups of $\Gal(\wh E/L)$.  So it suffices to show that for each $S'$, there is a point of $Z$ defined over the compositum $\til L \wh E^{S'} \subseteq \wh E$.

So consider any $S'$.  By Lemma~\ref{group_lemma} (which applies since $[\til E:\til L]$ is prime to $\ell$), $\til L \wh E^{S'}$ contains $\tau(\til E)$ for some $\tau \in \Gal(\wh E/L)$.  But $Z$ is an $L$-variety that has an $\til E$-point $\xi$, by Lemma~\ref{finite_exten_la}.  Hence $Z$
also has a $\tau(\til E)$-point, viz.\ $\tau(\xi)$.  
But $Z(\tau(\til E)) \subseteq Z(\til L \wh E^{S'})$, since $\tau(\til E) \subseteq \til L \wh E^{S'}$.
So indeed $Z$ has a point defined over $\til L \wh E^{S'}$.
\end{proof}

\begin{prop} \label{globalize_FP_extension}
Let $T$ be an excellent henselian discrete valuation ring with fraction field $K$.  Let $F$ be a one-variable function field over $K$ (e.g., a semi-global field, if $K$ is complete), and let $\XX$ be a normal model of $F$.  Let $P$ be a (not necessarily closed) point of the closed fiber of $\XX$, and let $Z$ be an $F$-scheme of finite type.
Let $\ell$ be a prime number, and suppose that there is a finite separable field extension $E_P/F_P$ of degree prime to $\ell$ such that $Z(E_P)$ is non-empty.  Then there is a finite separable field extension $E/F$ of degree prime to $\ell$ such that $Z(E \otimes_F F_P)$ is non-empty.  
\end{prop}

\begin{proof}
By Proposition~\ref{henselization}, there is a finite separable field extension $E_P'/F_P\h$ of degree prime to $\ell$ such that $Z(E_P')$ is non-empty.  Applying Proposition~\ref{abstract globalize prop} with $L=F$, $L'=F_P\h$, and $E'=E_P'$, we obtain a finite separable field extension $E/F$ of degree prime to $\ell$ such that $Z(E \otimes_F F_P\h)$ is non-empty.  But $F_P\h$ is contained in $F_P$, and so $Z(E \otimes_F F_P)$ is non-empty.
\end{proof}

\subsection{Local descent results}\label{local descent subsec}

In this subsection, we establish local descent results which will be used to prove local-global principles with respect to discrete valuations in Subsection~\ref{lgp dvr subsec}.

Let $A$ be a complete regular local ring of dimension 2 with field of fractions $F$ and residue field $k$. For any prime $\pi$ of $A$, let $F_\pi$ denote the completion of $F$ with respect to the discrete valuation associated to $\pi$, and let $k(\pi)$ denote the residue field.

The following lemma  is proved in (\cite[Lemma 5.1]{PPS}) for Galois extensions, and a similar proof gives the general case. 
  
\begin{lemma} \label{unramified}
Let $A$ be a complete regular local ring of dimension two, $F$ its field 
of fractions and
$k$ its residue field.  Let $\pi,\delta \in A$ generate the maximal 
ideal and let $E_\pi /F_\pi$ be a finite separable unramified field 
extension.
If $\cha(k)$ does not divide $[E_\pi : F_\pi]$,
then there exists a finite separable field extension $E/F$ such that $E 
\otimes_F F_\pi \simeq E_\pi$ and
the integral closure of $A$ in $E$ is a complete regular local ring with fraction field  $E$ and maximal ideal $(\pi', \delta')$, where $\pi'$ and $\delta'$ generate the unique primes lying over $\pi$ and $\delta$, respectively.
\end{lemma}

\begin{proof}  The residue field $F(\pi)$ of $F_\pi$  is  the field of 
fractions of $A/(\pi)$.
Since $A$ is a complete regular local ring, $A/(\pi)$ is a complete 
discrete valuation ring,
$k$ is the residue field of $A/(\pi)$, and the image $\bar{\delta}$ 
of $\delta$ is a uniformizer in $A/(\pi)$.

Let $E(\pi)$ be the residue field of $E_\pi$; this is a complete discretely valued field.
Let $L(\pi)$ be the maximal unramified field extension of $F(\pi)$ contained 
in $E(\pi)$, and let $L_\pi$ be the subextension of $E_\pi/F_\pi$ whose 
residue field is $L(\pi)$.
Let $\kappa$ be the residue field of $E(\pi)$ (or equivalently, of 
$L(\pi)$) at its discrete valuation.
Since $[E_\pi : F_\pi]$ is coprime to $\cha(k)$, so is $[\kappa:k]$, and 
thus $\kappa$ is a finite separable field extension of $k$.
Write $\kappa = k[t]/(f(t))$ for some monic separable polynomial $f(t) 
\in k[t]$. By lifting the polynomial $f(t) $ to a
monic polynomial over $A$, we obtain a finite \'etale $A$-algebra $B$; 
this is
a complete regular local ring with maximal ideal $(\pi, \delta)$ at 
which the
residue field is $\kappa$.
The residue field of $L$ (i.e., of $B$) at $\pi$ is $L(\pi)$. Since the 
same is true for $L_\pi$, it follows that the complete discretely valued 
fields $L \otimes_F F_\pi$ and $L_\pi$ are isomorphic over $F_\pi$. 
Note that the image $\bar{\delta} \in B/(\pi) \subset L(\pi)$ of 
$\delta$ is a uniformizer for $L(\pi)$.

Now $E(\pi)/L(\pi)$ is totally ramified of degree $d$ prime to 
char$(k)$.  So by
\cite[Proposition II.5.12]{Lang},
$E(\pi) = L(\pi)(\sqrt[d]{v\bar{\delta}})$ for some unit $v \in 
B/(\pi)$, the valuation ring of $L(\pi)$.
Let $u \in B$ be a lift of $v \in B/(\pi)$ and let $E= 
L(\sqrt[d]{u\delta})$.
Then $E$ is a finite separable field extension of $F$, and $E \otimes_F F_\pi 
\simeq E \otimes_L L \otimes_F F_\pi \simeq E \otimes_L L_\pi
\simeq L_\pi(\sqrt[d]{u\delta}) \simeq
E_\pi$.
The integral closure of $B$ in $E$ (or equivalently, of $A$ in $E$) is a 
two-dimensional complete local ring with maximal ideal
$(\pi,\sqrt[d]{u\delta})$, and hence it is regular. Its fraction field is $E$, and the ideals in this ring that are 
generated by $\pi$ and $\sqrt[d]{u\delta}$ are the unique prime ideals lying 
over the ideals $(\pi)$ and $(\delta)$ in $A$.  So $E$ is as asserted.
\end{proof}

  The above lemma can be used to obtain a descent statement for not necessarily unramified extensions.
   \begin{lemma} 
   \label{pi} Let $A$ be a complete regular local ring of dimension 2, $F$ its field of fractions and
 $k$ its residue field.  Let $\pi,\delta \in A$ generate the maximal ideal and let $E_\pi /F_\pi$ be a finite separable field extension.
Suppose $\cha(k)$ does not divide $[E_\pi : F_\pi]$. Let $\ell$ be a prime number.
If $[E_\pi : F_\pi]$ is prime to $\ell$,  then there exists a finite separable field extension $E/F$ such that  
\begin{itemize}
\item $[E : F]$ is prime to $\ell$;
\item $E \otimes_F F_\pi $ is a field; 
\item $E_\pi$ is isomorphic to a subfield of $E \otimes_F F_\pi $;
\item the integral closure of $A$ in $E$ is a complete regular local ring with fraction field  $E$ and maximal ideal $(\pi', \delta')$, where $\pi'$ and $\delta'$ generate the unique primes 
lying over $\pi$ and $\delta$, respectively.
\end{itemize}  \end{lemma}

 \begin{proof} Let $L_\pi/F_\pi$ be the maximal unramified extension contained in $E_\pi$.
 By Lemma~\ref{unramified}, there is a finite separable extension $L/E$ 
such that $L\otimes F_\pi \simeq L_\pi$; the integral closure of $A$ in 
$L$ is regular with maximal ideal $(\pi', \delta')$; and such that 
$(\pi')$ and $(\delta')$ are the unique primes lying over the primes 
$(\pi)$ and $(\delta)$ of $A$.
Thus, replacing $F_\pi$ by $L_\pi$ and $F$ by $L$, we may assume that 
 $E_\pi/F_\pi$ is totally ramified. Since $\cha(k)$ does not divide $n := [E_\pi : F_\pi]$, neither does 
the characteristic of the residue field $k(\pi)$ of $F_\pi$.  Hence 
$E_\pi = F_\pi(\sqrt[n]{u\pi})$ for some $u \in F_\pi$ which is a unit 
at the discrete valuation of $F_\pi$ (\cite[Section II.5, Proposition 
12]{Lang}). Let $\bar u$ be the image of $u$ in $k(\pi)$. Since $k(\pi)$ is the field of 
 fractions of $A/(\pi)$,  we have $\bar u =\bar{ v}\bar{\delta}^i$ for some $i$ and some unit $\bar{v} \in A/(\pi)$ with preimage $v \in A$. 
 Thus $u^{-1}v\delta^i$ lies in the valuation ring of $F_\pi$ and is 
congruent to $1$ mod $\pi$.  By Hensel's Lemma, this element has an 
$n$-th root in $F_\pi$, and so $E_\pi = F_\pi(\sqrt[n]{v\delta^i \pi})$.
 Let $E = F(\sqrt[n]{\delta}, \sqrt[n]{v\pi})$. Since $n$ is not divisible by either $\cha(k)$ or $\ell$,
the field extension $E/F$ is separable and of degree prime to $\ell$. 
Moreover, $E \otimes_F F_\pi$ is isomorphic to the field 
$F_\pi(\sqrt[n]{\delta}, \sqrt[n]{v\pi})$, which contains $E_\pi$ as a 
subfield. Since $E/F$ is finite 
and since $A$ is a complete local ring with fraction field $F$, the integral 
closure $B$ of $A$ in $E$ is a complete local ring with fraction field $E$. By (\cite[Lemma 3.2]{PS}), $B$ is a regular local ring with maximal ideal $( \sqrt[n]{\delta}, \sqrt[n]{v\pi})$. Since $(\sqrt[n]{\delta})$ and $(\sqrt[n]{v\pi})$ are the unique primes 
of $B$ lying over the primes $(\delta)$ and $(\pi)$ of $A$, the result 
follows.
 \end{proof}
 
Finally, we require a simultaneous descent result:
 
\begin{lemma} 
\label{pi-delta}
Let $A$ be a complete regular local ring of dimension~$2$, $F$ its field of fractions, $(\pi, \delta)$ its 
maximal ideal,  and  $k$ its residue field.  Suppose that $\cha(k) =0$.
Let $E_\pi/F_\pi$ and $E_\delta/F_\delta$ be finite field extensions.
Let $\ell$ be a prime. Suppose that  the degrees of $E_\pi/F_\pi$ and $E_\delta/F_\delta$ are prime to  
$\ell$. Then there exists a finite (separable) field extension $E/F$ such that 
\begin{itemize}
\item $[E : F]$ is prime to $\ell$; 
\item  the integral closure of $A$ in $E$ is a complete regular local ring;
\item $E\otimes_F F_\pi$ and  $E\otimes_F F_\delta$ are fields;
\item $E_\pi$ is isomorphic to a subfield of $E\otimes_F F_\pi$;  
\item $E_\delta$ is isomorphic to a subfield of $E\otimes_F F_\delta$.
\end{itemize}
 \end{lemma}
 
 \begin{proof} By Lemma~\ref{pi}, there exists a finite (separable) field extension $\tilde{E}/F$ such that 
 $[\tilde{E} : F]$ is prime to $\ell$, $\tilde{E} \otimes_F F_\pi$ is a field which contains an isomorphic copy of $E_\pi$ as a subfield, and
  the integral closure $\tilde{B}$ of $A$ in $\tilde{E}$ is a regular local ring with maximal ideal $(\pi', \delta')$, with
  $\pi'$ and $\delta'$ lying over $\pi$ and $\delta$, respectively. Moreover, the ideals $(\pi')$ and $(\delta')$ are uniquely determined by $\pi$ and $\delta$.
  Since $\delta'$ is a prime lying over $\delta$, $F_\delta \subseteq \tilde{E}_{\delta'}$. 
  Since $E_\delta/F_\delta$ is a separable field extension,
  $E_\delta \otimes_{F_\delta} \tilde{E}_{\delta'}$ is a product of field extensions $E_i/\tilde{E}_{\delta'}$. 
  Since  $[E_\delta : F_\delta]$ prime to $\ell$, $[E_i : \tilde{E}_{\delta'}]$ is prime to~$\ell$ for some~$i$.
 
  Again by Lemma~\ref{pi} (this time applied to the complete regular local ring $\til B$ 
and $E_i/\til E_{\delta'}$), there exists a finite field extension 
$E/\til{E}$ of degree prime to $\ell$ with the following properties:
$E \otimes_{\til{E}} \til{E}_{\delta'}$ is a field containing $E_i$ as a 
subfield; the integral closure $B$ of $\til{B}$ in $E$ is a complete 
regular local ring having fraction field $E$ and maximal ideal of the form
$(\pi'',\delta'')$; and $(\pi'')$ and $(\delta'')$ are the unique primes 
of $B$ that lie over the primes  $(\pi')$ and $(\delta')$ of $\til B$.
 Note that the uniqueness of $(\delta')$ implies that $\tilde{E}_{\delta'}=\tilde{E}\otimes_F F_\delta$.  Since $E_\delta$ is isomorphic to a subfield of $E_i$, $E_\delta$ is isomorphic to a subfield of $E \otimes_{\tilde{E}} \tilde{E}_{\delta'}=E\otimes_{\tilde E} (\tilde{E} \otimes_F F_\delta)=E\otimes_F F_\delta$ as claimed. Similarly, $E_\pi$ is a subfield of $\tilde{E}\otimes_F F_\pi$ which is a subfield of $E\otimes_FF_\pi$ by base change. The uniqueness of $(\pi'')$ implies that the latter is a field.
  Since $[E : \tilde{E}]$ and $[\tilde{E} : F]$ are prime to $\ell$, 
  $[E : F]$ is prime to $\ell$.  Since the integral closure of $A$ in  $E$ is $B$, the assertion follows.
 \end{proof}

\section{Local-global principles for zero-cycles}\label{lgp sec}

Let $F$ be a field, and let $Z$ be an $F$-scheme. 
A {\em zero-cycle} on $Z$ is a finite ${\mathbb Z}$-linear combination $\sum n_iP_i$ of closed points $P_i$ of $Z$. Its {\em degree} is $\sum n_i\operatorname{deg}(P_i)$, where $\operatorname{deg}(P_i)$ is the degree of the residue field of $P_i$ over $F$. We say that a closed point $P$ of $Z$ is  {\em separable} if its residue field is separable over $F$.  A zero-cycle $\sum n_iP_i$ is called {\em separable} if each $P_i$ is.

We will prove local-global principles for zero-cycles in different settings. First, we consider  collections of overfields coming from patching and from points on the closed fiber of a model. Second, we will use this to obtain local-global principles with respect to discrete valuations. Many of the statements in this section rely on the descent results in the previous section. In our results, we have to assume a local-global principle for points (in the respective setting) over the function field $F$ as well as over all of its finite separable field extensions. This is analogous to results in the number field case; see \cite{liang}, Prop.~3.2.3, and \cite{CTAWS}, Section~8.2. Example~\ref{counterexample_F} shows that this hypothesis is actually necessary. Corollary~\ref{LGP ex} below exhibits situations in which the assumption is satisfied. 

\subsection{Local-global principles with respect to patches and points}\label{lgp patches points subsec}
In this subsection we prove that certain local-global principles for the existence of rational points imply analogous local-global principles for zero-cycles, for varieties over semi-global fields.  This will be proven in two contexts, one where the overfields come from patching, and one where they correspond to points on the closed fiber of a model.

\begin{notation} \label{geom notn}
Let $\XX$ be a normal model of a one-variable function field $F$ over a discretely valued field $K$,
and let $X$ denote the closed fiber. Let $\mc P$ be a finite nonempty set of closed points of $X$ that meets each irreducible component of $X$, and let $\mc U$ be the set of connected components of the complement of $\mc P$ in $X$.  Let $\mc B$ be the set of branches of $X$ at points of $\mc P$. 
We let $\Omega_{\XX}$ be the collection of field extensions $F_P/F$ where $P$ is a (not necessarily closed) point of $X$, and let $\Omega_{\XX,\mc P}$ denote the collection of field extensions $F_\xi/F$ where $\xi\in \mc P \cup \mc U$. (See Notation~\ref{basic notation}.)
\end{notation}

\begin{notation}\label{extension notn}
Let $\XX$ be a normal model for $F$ as above, and let $E/F$ be a finite field extension. We let $\XX_E$ denote the normalization of $\XX$ in $E$. (This is a normal model for $E$.)
If $\mc P$ is a finite set of closed points of $\XX$, we let $\mc P_E$ denote its preimage under the natural map $\XX_E\rightarrow \XX$.
\end{notation}

\begin{dfn} \label{LGP def}
Let $F$ be a field, let $Z$ be an $F$-scheme of finite type, and let $\Omega$ be a collection of overfields of $F$.  We use the following terminology:
\begin{itemize}
\item The pair $(Z,\Omega)$ {\em satisfies a local-global principle for rational points} if  it has the following property: $Z(F)\neq \varnothing$ if and only if $Z(L)\neq \varnothing$ for every $L\in \Omega$.
\item The pair $(Z,\Omega)$ {\em satisfies a local-global principle for closed points of degree prime to $\ell$} if it has the following property: $Z$ has a closed point of degree prime to $\ell$ if and only if the base change $Z_L$ has a closed point of degree prime to $\ell $ for every $L\in \Omega$.
\item The pair $(Z,\Omega)$ {\em satisfies a local-global principle for zero-cycles of degree one} if and only if it has the following property: $Z$ has a zero-cycle of degree one if and only if $Z_L$ has a zero-cycle of degree one for all $L \in \Omega$.
\end{itemize}
Similarly, one can speak of a local-global principle for separable closed points or separable zero-cycles.
\end{dfn}

Below we consider the case where $F$ is a semi-global field (i.e., $K$ is complete).  Note that a 
finite separable field extension $E/F$ is again a semi-global field.

\begin{thm} \label{sep LGP prime patches}
Let $F$ be a semi-global field with normal model $\XX$.  Let $\mc P$ be a finite nonempty set of closed points that meets every irreducible component of the closed fiber $X$ of $\XX$.  Let $Z$ be an $F$-scheme of finite type. Assume that for all finite separable field extensions $E/F$, 
$(Z_E,\Omega_{\XX_E,\mc P_E})$ satisfies a local-global principle for rational points.   
Then for every prime number $\ell$, $(Z,\Omega_{\XX,\mc P})$ satisfies a local-global principle for separable closed points of degree prime to~$\ell$.
\end{thm}

\begin{proof}
Let $\mc U$ be as in Notation~\ref{geom notn}. We need to prove that if $Z_{F_\xi}$ has a separable closed point $z_\xi$ of degree prime to $\ell$ for every 
$\xi \in \mc P\cup \mc U$, then $Z$ has a separable closed point $z$ of degree prime to $\ell$. For each $\xi$, let $E_\xi/F_\xi$ be the residue field extension at~$z_\xi$. 

By Proposition~\ref{component to global}, for each $U\in \mc U$, there is a finite separable field extension $A_U$ of $F$ that induces $E_U$ (i.e., $A_U\otimes_F F_U$ is isomorphic to $E_U$),
and hence has degree prime to $\ell$. In particular, each $Z(A_U\otimes_F F_U)$ is non-empty.
By Proposition~\ref{globalize_FP_extension}, for each $P\in \mc P$, there is a finite separable field extension $A_P$ of $F$ of degree prime to $\ell$ such that $Z(A_P\otimes_F F_P)$ is non-empty.  
Let $A$ be the tensor product over $F$ of all the fields $A_P$ and $A_U$.  
This is an \'etale $F$-algebra of degree prime to $\ell$, and it is the direct product of finite separable field extensions $A_i$ of $F$, each of which is a compositum of the fields $A_P$ and $A_U$.  
Since the degree of $A$ is the sum of the degrees of the fields $A_i$ over $F$, at least one of those fields 
$A_i$ has degree prime to $\ell$.  
Write $E$ for this field.  So $E/F$ is separable of degree prime to $\ell$, and for each $\xi \in \mc P\cup \mc U$, $Z(E \otimes_F F_\xi)$ is non-empty. 

Let $\XX_E$, $\mc P_E$ be as in Notation~\ref{extension notn}, let $X_E$ be the closed fiber, and  let $\mc U_E$ be the set of connected components of the complement of $\mc P_E$ in $X_E$. For each $P \in \mc P$, $E \otimes_F F_P$ is the direct product of the fields $E_{P'}$, where
$P'$ runs over the points of $\mc P_E$ that lie over $P$; and similarly for each $U \in \mc U$.  Hence for each $\xi'\in P_E \cup \mc U_E$, $Z(E_{\xi'})$ is non-empty. By assumption, this implies $Z(E)$ is nonempty; 
i.e., $Z$ has a point defined over a finite separable field extension of $F$ of degree prime to $\ell$.
\end{proof}

The hypothesis in the above theorem cannot be weakened to consider merely $F$ instead of all finite separable field extensions; see Example~\ref{counterexample_F} below.

Given a variety $V$ over a field $k$, the \textit{index} (resp.\ \textit{separable index}) of $V$ is the greatest common divisor of the degrees of the finite (resp.\ finite separable) field extensions of $k$ over which $V$ has a rational point.  This is the same as the smallest positive degree of a zero-cycle (resp.\ separable zero-cycle) on $V$.  In this terminology, Theorem~\ref{sep LGP prime patches} says that if the separable index of each $Z_{F_\xi}$ is prime to $\ell$, then so is the separable index of $Z$; i.e.,\ if each $Z_{F_\xi}$ has a zero-cycle of degree prime to $\ell$, then so does $Z$.

\begin{cor} \label{sep LGP patches}
Let $F$ be a semi-global field with normal model $\XX$, and let $X$ denote the closed fiber. Let $\mc P$ be a finite nonempty set of closed points of $X$ which meets every irreducible component of $X$. Let 
$Z$ be an $F$-scheme of finite type such that for every finite separable extension $E/F$, $(Z_E,\Omega_{\XX_E,\mc P_E})$ satisfies a local-global principle for rational points. Then the prime numbers that divide the separable index of $Z$ are precisely those that divide the separable index of some $Z_{L}$ for $L\in \Omega_{\XX,\mc P}$.  In particular, the separable index of $Z$ is equal to one if and only if the separable index of each $Z_{L}$ is equal to one. 
\end{cor}

\begin{cor} \label{LGP patches}
In Corollary~\ref{sep LGP patches}, if $\cha F=0$, or if $Z$ is regular and generically smooth, then the assertion also holds with the separable index replaced by the index. In particular, $(Z,\Omega_{\XX,\mc P})$ satisfies a local-global principle for zero-cycles of degree one.
\end{cor}

\begin{proof}
This is trivial in the former case.  In the latter case it follows from the fact that under those hypotheses the index is equal to the separable index (Theorem~9.2 of \cite{GLL}).
\end{proof}

The above results (and their proofs) show that if $(Z_E,\Omega_{\XX_E,\mc P_E})$ satisfies a local-global principle for rational points for all finite separable $E/F$, the index (resp.\ separable index) of $Z$ divides some power of the least common multiple of the indices (resp.\ separable indices) of the fields $Z_{L}$ for $L\in \Omega_{\XX,\mc P}$. It would be interesting to obtain a bound on the exponent, in terms of~$F$ and~$Z$.  (See Section~2.2 of \cite{CTHHKPS:Kampen} for a partial result in this direction.)

\begin{example}\label{counterexample_F}
This example shows that in Theorem~\ref{sep LGP prime patches} and in Corollaries~\ref{sep LGP patches}
and~\ref{LGP patches},
we cannot simply assume that $Z$ satisfies a local-global principle over $F$, but must also consider finite field extensions of $F$.  Let $k$ be a field of characteristic zero, let $T=k[[t]]$, let $K =k((t))$, let $F=K(x)$, and let $\XX$ be the projective $x$-line over $T$.  
On the closed fiber of $\XX$, take $\mc P$ to be the set 
$\{P_0,P_1\}$ consisting of the two points on the closed fiber where $x=0$ and where $x=1$.  The complement of $\mc P$ in the closed fiber is a connected affine open set $U$; take $\mc U = \{U\}$. 

Let $E = F[y]/(y^2-x(x-t)(x-1)(x-1-t))$, 
a degree two separable field extension of $F$ that splits over $F_U$ but whose base change to $F_{P_i}$ is a degree two field extension for $i=0,1$.  So the (separable) index of $Y := \Spec(E)$ 
is two over $F$ and over each $F_{P_i}$, but is one over $F_U$.
The normalization $\mc Y$ of $\mc X$ in $E$ is a degree two branched cover of $\XX$, whose
fiber over the generic point of $\XX$ is $Y$, and whose fiber 
over the closed point of $\Spec(T)$ consists of two copies of the projective $k$-line that meet at two points.  
Its reduction graph contains a loop, and so by \cite[Proposition~6.2]{HHK:H1} there is a degree three connected {\em split cover} $\ZZ \to \YY$ (i.e., $\ZZ \times_{\YY} Q$ consists of three copies of $Q$ for every point $Q \in \YY$ other than the generic point).  The fiber $Z$ of $\ZZ$ over the generic point of $\XX$ is the spectrum of a degree six separable field extension $L$ of $F$, and hence the index of $Z$ over $F$ is six.  
But since $\ZZ \to \YY$ is a split cover, the index of $Z$ over each $F_{P_i}$ is two and over $F_U$ is one (the same as for $Y$).

Since the index of $Z$ over $F$ (resp.\ over $F_{P_i}$, $F_U$) is six 
(resp.\ two, one), it follows that $Z(F_{P_i})=\emptyset$ and hence $(Z,\Omega_{\mc X,\mc P})$ trivially
satisfies a local-global principle for rational points; but the conclusions of Theorem~\ref{sep LGP prime 
patches} and its corollaries fail here if we take $\ell=3$.  The explanation 
is that $(Z_E,\Omega_{\mc Y, \mc P_E})$ does not satisfy a local-global principle for rational 
points, due to the above index 
computations for $Z$.
\end{example}

\medskip

We next prove analogous results in which {\em all} the points on the closed fiber $X$ of $\XX$ are used.  
That is, instead of considering a collection of overfields of the form $\Omega_{\XX,\mc P}$ for some finite set $\mc P$ as before, we consider the collection $\Omega_{\XX}$ of all overfields of $F$ of the form $F_P$, where $P$ ranges over all (not necessarily closed) points of $X$.

\begin{prop} \label{LGP equiv}
Let $F$ be a semi-global field with normal model $\XX$, and let $\mc P$ be a finite nonempty set of closed points of $\XX$ which meets every irreducible component of the closed fiber. Let $Z$ be an $F$-scheme of finite type.  Then the following are equivalent: 
\begin{enumerate}
\item $(Z,\Omega_{\XX})$ satisfies a local-global principle for rational points.
\item For every choice of $\mc P$ as in Notation~\ref{geom notn}, $(Z,\Omega_{\XX,\mc P})$ satisfies a local-global principle for rational points.
\end{enumerate}
\end{prop}

\begin{proof}
First suppose that $(Z,\Omega_\XX)$ satisfies a local-global principle for rational points. 
Let $\mc P$ and $\mc U$ be as in Notation~\ref{geom notn}.  Assume that $Z$ has a rational point over $F_\xi$ for each 
$\xi \in \mc P\cup \mc U$.  Then every point $P$ of $X$ that is not contained in $\mc P$
lies in some $U \in \mc U$, and hence $F_U \subseteq F_P$.  Thus $Z$ has a rational point over $F_P$ for every point $P \in X$ (including the generic points of irreducible components).  Hence by hypothesis, $Z$ has an $F$-point.  Thus $(Z,\Omega_{\XX,\mc P})$ satisfies a local-global principle for rational points.

Conversely, suppose that $(Z,\Omega_{\XX,\mc P})$ satisfies a local-global principle for rational points for every choice of $\mc P$ as above. 
Assume that $Z(F_P)$ is non-empty for every (not necessarily closed)
point $P \in X$.  Given an irreducible component $X_i$ of $X$, consider its generic point $\eta_i$.  
By Proposition~5.8 of \cite{HHK:H1}, since $Z(F_{\eta_i})$ is non-empty
it follows that $Z(F_{U_i})$ is also non-empty for some non-empty affine open subset $U_i \subset X_i$ that does not meet any other irreducible component of $X$.  
Let $\mc U$ be the collection of these disjoint open sets $U_i$ of $X$, and let $\mc P$ be the complement of their union in $X$; note that this is a non-empty finite set of points. 
By hypothesis, $Z$ has an $F$-point (note that $Z(F_P)$ is non-empty for each $P \in \mc P$).  So  $(Z,\Omega_{\XX})$ satisfies a local-global principle for rational points.
\end{proof}

Using Proposition~\ref{LGP equiv}, we obtain the following result, which parallels Theorem~\ref{sep LGP prime patches}, Corollary~\ref{sep LGP patches}, and Corollary~\ref{LGP patches}:

\begin{thm} \label{LGP closed fiber}
Let $F$ be a semi-global field with normal model $\XX$. Let $Z$ be an $F$-scheme of finite type such that for every finite separable field extension $E/F$, $(Z_E,\Omega_{\XX_E})$ satisfies a local-global principle for rational points.
Then
\renewcommand{\theenumi}{\alph{enumi}}
\begin{enumerate}
\item\label{prime} For every prime number $\ell$, $(Z,\Omega_{\XX})$ satisfies a local-global principle for separable closed points of degree prime to $\ell$.
\item\label{prime sep index} The prime numbers that divide the separable index of $Z$ are precisely those that divide the separable index of $Z_L$ for some $L \in \Omega_{\XX}$. In particular, the separable index of $Z$ is equal to one if and only if the separable index of each $Z_L$ is equal to one.
\item\label{prime index} If $\cha(F)=0$ or if $Z$ is regular and generically smooth, the previous assertion also holds with the separable index replaced by the index. In particular, $(Z,\Omega_{\XX})$ satisfies a local-global principle for zero-cycles of degree one.
\end{enumerate}
\end{thm}

\begin{proof} We begin by proving~(\ref{prime}). Suppose that for every $P$ in the closed fiber $X$ of $\XX$, $Z$ has a point over a finite 
separable extension $E_P$ of $F_P$ having degree prime to $\ell$.  For 
each irreducible component $X_i$ of $X$ with generic point $\eta_i$,
Proposition~\ref{globalize_FP_extension} yields a finite separable field 
extension $E_i$ of $F$ of degree prime to $\ell$ such that $Z(E_i 
\otimes_F F_{\eta_i})$ is non-empty.
By Proposition~5.8 of \cite{HHK:H1}, there is a non-empty affine open 
subset $U_i \subset X_i$ that meets no other irreducible component of 
$X$, such that $Z(E_i \otimes_F F_{U_i})$ is non-empty.  Since $E_i/F$ 
is a separable field extension of degree prime to $\ell$, so is 
$E_{U_i}/F_{U_i}$ for some direct factor $E_{U_i}$ of $E_i \otimes_F 
F_{U_i}$.   Here $Z(E_{U_i})$ is non-empty, since $Z(E_i \otimes_F 
F_{U_i})$ is non-empty.

As in the proof of Proposition~\ref{LGP equiv}, we let $\mc U$ be the 
collection of open sets $U_i$, and let $\mc P$ consist of the (finitely 
many) closed points of $X$ that do not lie in any $U_i$.
For any finite separable field extension $E/F$,
$(Z_E,\Omega_{\XX_E,\mc P_E})$ satisfies a local-global principle for 
rational points,
by the hypothesis of the theorem together with Proposition~\ref{LGP 
equiv} applied to $Z_E$.  So by Theorem~\ref{sep LGP prime patches},
$(Z,\Omega_{\XX,\mc P})$ satisfies a local-global principle for separable closed
points of degree prime to $\ell$.  Since $Z$ has a separable point of degree prime
to $\ell$ over $F_\xi$ for each $\xi \in \mc P \cup \mc U$, it follows
that $Z$ has a separable point of degree prime to $\ell$.

The other two statements are then immediate, as in the analogous case for $\Omega_{\XX,\mc P}$ in Corollaries~\ref{sep LGP patches} and~\ref{LGP patches}.
\end{proof}

The following corollary exhibits some cases in which our theorem above applies, for homogeneous spaces under linear algebraic groups that are connected and (retract) rational.  Here, by a {\em homogeneous space} $Z$ under a linear algebraic group $G$ defined over a field $F$, we mean an $F$-scheme $Z$ together with a group scheme action of $G$ on $Z$ over $F$ such that the action of the group $G(\bar F)$ on the set $Z(\bar F)$ is simply transitive.

\begin{cor} \label{LGP ex}
Let $F$ be a semi-global field with normal model $\XX$.
Let $G$ be a linear algebraic group over $F$ (i.e., a smooth affine group scheme over $F$), and let $Z$ be a homogeneous space for $G$. In each of the following cases, the prime numbers that divide the separable index of $Z$ are exactly those that divide the separable index of some $Z_L$, where $L\in \Omega_{\XX}$; and moreover $(Z,\Omega_{\XX})$ satisfies a local-global principle for separable zero-cycles of degree one.
\begin{enumerate}
\item\label{retract} $G$ is connected and retract rational, and $Z$ is a torsor under $G$.
\item\label{trans} $G$ is connected and rational, and $G(E)$ acts transitively on $Z(E)$ for every field extension $E/F$.
\end{enumerate}
The same conclusion holds without separability (i.e., for the index and for zero-cycles) if $Z$ is smooth; e.g., this holds in case~(\ref{retract}).
\end{cor}

 \begin{proof}
In both cases, $(Z,\Omega_{\XX})$ satisfies a local-global principle for rational points. 
In the first case, this was shown in \cite{Kra} (see also Corollary~6.5 of \cite{HHK:H1} for the case when $G$ is rational);  in the second this follows from Corollary~2.8 of \cite{HHK:H1}.
These local-global principles also hold for $Z_E$ if $E/F$ is a finite separable field extension of $F$, since the above assumptions on $F$ are preserved under such a base change. 
Thus Theorem~\ref{LGP closed fiber}(\ref{prime sep index}) applies.  
If $Z$ is smooth (e.g., in case (\ref{retract})), Theorem~\ref{LGP closed fiber}(\ref{prime index}) also applies.
\end{proof}

Theorem~\ref{LGP closed fiber} also applies in the 
case of certain other linear algebraic groups that need not be rational; 
see Section~4.3 of [HHK14] for examples.

\subsection{Local-global principles over excellent henselian discrete valuation rings}\label{lgp henselian subsec}

In this subsection, $T$ is assumed to be excellent henselian instead of complete. We now extend the results from the previous subsection to that case via Artin Approximation.

Let $T$ be an excellent henselian discrete valuation ring, with residue field $k$ and fraction field $K$.  Let $\XX$ be a normal flat projective $T$-curve with closed fiber $X$ and function field $F$; i.e.,  $\XX$ is a normal model of $F$.  Every finite separable field extension of $F$ is also the function field of a normal flat projective $T$-curve.  The completion $\wh T$ of $T$ is a complete discretely valued field with the same residue field $k$; its fraction field $\wh K$ is the completion of $K$.  Because $T$ is henselian, the base change $\wh{\XX} := \XX \times_T \wh T$ is a normal connected projective $\wh T$-curve with closed fiber $X$ and function field $\wh F := \Frac(F \otimes_K \wh K)$.  The complete local rings of $\wh{\XX}$ and $\XX$ at a point $P \in X$ are naturally isomorphic, and so we may write $\wh R_P$ and $F_P$ without ambiguity, for those rings and their fraction fields. 
Similarly, for $U$ an affine open subset of the closed fiber, 
the notations $\wh R_U$ and $F_U$ are unambiguous, being the same for $\wh{\XX}$ and $\XX$.
If $E/F$ is a finite separable field extension, we write 
$\wh E = E \otimes_F \wh F = \Frac(E \otimes_K \wh K)$.  This is the function field of the normalization $\XX_E$ of $\XX$ in $E$, whose closed fiber is denoted by $X_E$.

\begin{prop} \label{henselian pt la}
In the above situation, let $Z$ be an $F$-scheme of finite type, let $E/F$ be a finite separable field extension.
If $Z(\wh E)$ is non-empty then so is $Z(E)$.  
\end{prop}

\begin{proof}
By hypothesis, $Z_E$ has a point over $\wh E$.  By Artin's Approximation Theorem (Theorems~1.10 and~1.12 of \cite{artin}), $Z_E$ also has an $E$-point.  Equivalently, $Z$ has an $E$-point.
\end{proof}

Let $F$ be as above with normal model $\XX$, and let $Z$ be an $F$-scheme of finite type. 
As in Definition~\ref{LGP def}, $\Omega_{\XX}$ denotes the set of all overfields of the form $F_P$. 
By Proposition~\ref{henselian pt la} above, $(Z,\Omega_{\XX})$ satisfies a local-global principle for rational points if and only if $(\wh Z,\Omega_{\wh\XX})$ satisfies such a local-global principle, where $\wh Z:= Z \times_F \wh F$. 

Preserving the notation given just before Lemma~\ref{henselian pt la}, we have:

\begin{prop} \label{LGP cl fiber henselian}
The assertions of Theorem~\ref{LGP closed fiber}
remain true if $F$ is the function field of a curve over an excellent henselian discrete valuation ring $T$ as above. 

In particular, assume that $Z$ is an $F$-scheme of finite type which is regular and generically smooth, such that for every finite separable field extension $E/F$, $(Z_E,\Omega_{\XX_E})$ satisfies a local-global principle for rational points. Then $(Z,\Omega_{\XX})$ satisfies a local-global principle for zero-cycles of degree one.
\end{prop}

\begin{proof}
Let $\ell$ be prime and assume that for every $P \in X$ there is a
finite separable extension $E_P/F_P$ of degree prime to $\ell$ such that
$Z(E_P)$ is non-empty.  By Proposition~\ref{globalize_FP_extension},
for each $P$ there is a finite separable field extension $A_P/F$ of
degree prime to $\ell$ such that $Z(A_P \otimes_F F_P)$ is non-empty.
If $\eta$ is the generic point of an irreducible component of $X$, then
by applying [HHK15a, Proposition 5.8] to $\wh Z_{A_\eta}$ we deduce that
$Z(A_\eta \otimes_F F_U) = \wh Z(A_\eta \otimes_F F_U)$ is non-empty for
some dense open subset $U$ of that component that meets no other
component.  Write $A_U = A_\eta$.  Let $\mc U$ be the collection of
these (finitely many) sets $U$, and let $\mc P$ be the complement in $X$
of the union of these sets $U$.

Proceeding as in the second half of the proof of Theorem~\ref{sep LGP prime patches} using the
field extensions $A_P, A_U$ of $F$ (for $P \in \mc P$ and $U \in \mc
U$), we obtain a finite separable extension $E/F$ of degree prime to
$\ell$ such that $Z(E \times_F F_\xi)$ is non-empty for each $\xi \in
\mc P \cup \mc U$.  Now every point $P \in X$ outside of $\mc P$ lies on
some $U \in U$, and hence $F_P$ contains $F_U$; so $Z(E \times_F F_P)$
is non-empty for every point $P$ on the closed fiber of $\XX$.  Thus by hypothesis, $Z(E)$ is
non-empty.  This proves the analog of Theorem~\ref{LGP closed fiber}(\ref{prime}) in the henselian case.

The henselian analogs of the remaining assertions are then immediate, as before.
\end{proof}

As a consequence, Corollary~\ref{LGP ex}
remains valid if $F$ is the function field of a curve over an excellent henselian discrete valuation ring.

In the above situation, with $T$ an excellent henselian discrete 
valuation ring and $P$ a point of $\XX$, let $R_P\h$ be the 
henselization of the local ring $R_P$ at its maximal ideal, and let 
$F_P\h$ be its fraction field.  Also, for every discrete valuation $v$ 
on $F$, let $R_v \subset F$ be the associated valuation ring, with 
henselization $R_v\h$ and completion $\wh R_v$, having fraction fields 
$F_v\h$ and $F_v$.  Recall that a
point $Q$ of $\XX$ is called the \textit{center} of $v$ if $R_v$ 
contains the
local ring $\mc O_{Z,Q}$ of $Q$ on $Z$, and if the maximal ideal 
${\mathfrak m}_Q$ of $\mc O_{Z,Q}$ is the contraction of the maximal 
ideal ${\mathfrak m}_v$ of $R_v$.

\begin{prop} \label{field containment}
Let $T$ be an excellent henselian discrete valuation ring, and let $F$ 
and $\XX$ be as above.   For every discrete valuation $v$ on $F$, 
there is a point $P$ on the closed fiber $X$ of $\XX$ such that $F_P\h 
\subseteq F_v\h$.
\end{prop}

\begin{proof}
We proceed as in the proof of \cite[Proposition~7.4]{HHK:H1}.  First 
note that $v$ has a center $Q$ on $\XX$ that is not the generic point 
of $\XX$.  In the case that $T$ is complete, this was shown in 
\cite[Lemma~7.3]{HHK:H1}; but the proof used the completeness hypothesis 
only to cite Hensel's Lemma in the proof of Lemma~7.1 there, and that 
holds by definition for henselian rings.
If $Q$ lies on $X$, then we may take $P=Q$.  Otherwise, $Q$ is a closed 
point of the generic fiber of $\XX$, of codimension one, and $v$ is 
defined by $Q$; so $F_Q = F_v$.  Since $T$ is henselian, the closure of 
$Q$ in $\XX$ meets $X$ at a single closed point $P$.  Every \'etale 
neighborhood of $P$ in $\XX$ also defines an \'etale neighborhood of 
$Q$, since the residue field of $Q$ is a henselian discretely valued 
field whose residue field corresponds to $P$.  Thus $R_P\h \subseteq 
R_Q\h = R_v\h$, and hence $F_P\h \subseteq F_v\h$.
\end{proof}

\subsection{Local-Global Principles with respect to discrete valuations in characteristic $0$}\label{lgp dvr subsec}

In the previous subsections, we considered local-global principles with respect to a collection of overfields arising from patching. More classically, local-global principles are stated with respect to overfields that are completions at discrete valuations. The aim of this subsection is to discuss when such local-global principles for rational points imply analogous principles for zero-cycles. Moreover, we show that in some cases, the existence of local zero-cycles of degree one even implies the existence of a global rational point.

\begin{notation}
For a field $F$, let $\Omega_F$ be the collection of overfields of $F$ that are completions of $F$ at discrete valuations.
\end{notation}

Below we consider the function field $F$ of a curve over an excellent henselian discrete valuation ring.  In this situation, we begin by explaining how local-global principles with respect to discrete valuations relate to local-global principles as studied in the previous subsections.

\begin{prop} \label{dvr lgp implies pt lgp}
Let $T$ be an excellent henselian discrete valuation ring, and let $F$ be the function field of a normal flat projective $T$-curve
$\XX$, with closed fiber $X$.  Let $Z$ be an $F$-scheme of finite type. 
If $(Z,\Omega_F)$ satisfies a local-global principle for rational points then $(Z,\Omega_{\XX})$ satisfies a local-global 
principle for rational points.
\end{prop}

\begin{proof}
Suppose that $Z$ has an $F_P$-point for every point $P \in X$.  We want 
to show that $Z$ has an $F$-point.
By Proposition~\ref{field containment}, for every discrete valuation $v$ 
on $F$, there is a point $P \in X$ such that $F_P\h \subseteq F_v\h$. 
By Artin's Approximation Theorem (\cite[Theorem~1.10]{artin}) applied to 
the coordinate ring of an affine neighborhood of $P$ in $\XX$, $Z$ has 
an $F_P\h$-point.  Hence $Z$ has an $F_v\h$-point, and thus also an 
$F_v$-point.  By the local-global hypothesis, $Z$ has an $F$-point.
\end{proof}

We next consider the case of torsors under a linear algebraic group $G$ over a function field $F$. These are classified up to isomorphism by the first Galois cohomology set $H^1(F,G)$. We define $\Sha(F, G)$ to be the kernel of the local-global map $H^1(F,G)\rightarrow \prod\limits_vH^1(F_v,G)$, where $v$ runs over all discrete valuations on~$F$. Hence $\Sha(F,G)$ is trivial if and only if $(Z,\Omega_F)$ satisfies a local-global principle for rational points for every $G$-torsor $Z$.

\begin{thm}\label{LGPdvr}
Let $T$ be an excellent henselian discrete valuation ring, $K$ its field of fractions and
$k$ its residue field. Suppose that $\cha(k)  = 0$. Let $F$ be a one variable function field over $K$ and let $\XX$  be a regular 
model of $F$ over $T$.  Let $G$ be a connected  linear 
algebraic group over $F$ which is the generic fiber of a reductive smooth group scheme over $\XX$. Let $Z$ be a $G$-torsor over $F$. Suppose that for all finite separable field extensions $E/F$, $(Z_E,\Omega_E)$ satisfies a local-global principle for rational points. Then $(Z,\Omega_F)$ satisfies a local-global principle for zero-cycles of degree one.
In particular, this applies if $\Sha(E, G_E)$ is trivial for all finite separable field extensions $E/F$.   
 \end{thm}
 
 \begin{proof} Let $Z$ be a $G$-torsor over $F$. Let $\YY$ be a sequence of blow-ups of $\XX$
 such that $Z$ is unramified except along a union of regular curves with normal crossings, i.e., $Z$ extends to a torsor on the complement of the union of these curves. 
 Since the underlying group scheme of~$G$ is smooth over $\XX$,  it is also smooth over $\YY$. 
 Let $Y$ be the closed fiber of $\YY$. 
 
 Let  $P \in \YY$ be a closed point. 
 Let $R_P$ be the regular local ring at $P$ and $\wh{R}_P$ be the completion of 
 $R_P$ at its maximal ideal.  Since the ramification locus of $Z$ is a union of regular curves 
 with normal crossings, there exist $\pi, \delta \in R_P$ such that the maximal ideal at $P$
 is $(\pi, \delta)$ and $Z$ is unramified on $R_P$ except possibly at $\pi$ and $\delta$. 
 
 Suppose that $Z$ has a zero-cycle of degree one over $F_v$ for all discrete valuations $v$ of $F$. 
 Let $\ell$ be a prime. Then there exist field extensions $E_\pi/F_{P,\pi}$ and $E_\delta/F_{P,\delta}$
of degree prime to $\ell$ such that $Z(E_\pi) \neq \emptyset$ and $Z(E_\delta) \neq \emptyset$. Here $F_{P,\pi}:=(F_P)_\pi$ and $F_{P,\delta}:=(F_P)_\delta$ are associated to $F_P$ as in the beginning of Subsection~\ref{local descent subsec}.
By Lemma~\ref{pi-delta}, there exists a field extension $E_P/F_P$ of degree prime to $\ell$ such that 
the integral closure $B_P$ of $\wh{R}_P$ in $E_P$ is a complete regular local ring and
$E_\pi$ (resp. $E_\delta$) is isomorphic to a subfield of the field $E_P \otimes_{F_P} F_{P, \pi}$ (resp. $E_P \otimes_{F_P} F_{P,\delta}$).
Moreover, the maximal ideal of $B_P$ is of the form $(\pi' , \delta')$ for unique primes $\pi'$ and $\delta'$ lying over 
$\pi$ and $\delta$, respectively. 

Since $Z(E_\pi)$ and $Z(E_\delta)$ are  non-empty, $Z(E_P \otimes F_{P, \pi})$ and $Z(E_P \otimes F_{P, \delta})$ are
 non-empty.  Hence $Z_{E_P \otimes F_{P, \pi}} \simeq G_{E_P \otimes F_{P, \pi}}$
 and $Z_{E_P \otimes F_{P, \delta}} \simeq G_{ E_P \otimes F_{P, \delta}}$.
In particular $Z$ is unramified at  $\pi'$ and $\delta'$. 
But $Z$ is unramified on $R_P$ except possibly at $\pi$ and $\delta$, thus it is indeed unramified at every height one prime ideal 
of $B_P$.  Hence the class of $Z_{E_P}$ comes from a class $\zeta$  in $H^1(B_P,  G)$ (\cite[Corollary 6.14]{CTS}). 
Since $\zeta$ is trivial over the completion  at $\pi'$, the image  of $\zeta$ in $H^1(k(\pi'), G)$
is trivial, where $k(\pi')$ is the residue field of $E_\pi$. Since  $\zeta \in H^1(B_P, G)$,  the image of $\zeta$ in $H^1(k(\pi'), G)$ comes from the image $\zeta(\pi')$ 
of $\zeta$ in $H^1(B_P/(\pi'), G)$. Since $B_P/(\pi')$ is a discrete valuation ring with field of fractions $k(\pi')$ and
the image of $\zeta(\pi')$ in $H^1(k(\pi'), G)$ is trivial, $\zeta(\pi')$ is trivial (\cite{Nis}). Hence 
the image of $\zeta$ in $H^1(\kappa, G)$ is trivial, where $\kappa$ is the residue field of $B_P$.
Since $B_P$ is a complete regular local ring, Hensel's Lemma implies that $\zeta$ is trivial in $H^1(B_P, G)$, and hence 
its image is trivial in $H^1(E_P, G)$. Thus $Z(E_P)\neq \emptyset$. Since $\operatorname{gcd}(\ell,[E_P:F_P])=1$ and since the prime $\ell$ is arbitrary, $Z$ has a zero-cycle of degree one over $F_P$. 

Since $(Z,\Omega_F)$ satisfies a local-global principle for rational points, $(Z,\Omega_{\YY})$ satisfies a local-global principle for rational points by Proposition~\ref{dvr lgp implies pt lgp}, and similarly for finite separable field extensions $E/F$.
Since $Z$ has a zero-cycle of degree one over $F_P$ for all $P \in Y$, 
$Z$ has a zero-cycle of degree one over $F$ by Proposition~\ref{LGP cl fiber henselian}.  
\end{proof}

\begin{remark}\label{weaker hyp remark}
As the proof shows, the above theorem holds under the weaker assumption that for every blow-up (i.e., birational projective 
morphism) $\mc Y \to \mc X$, and for every finite separable field 
extension $E/F$, $(Z_E,\Omega_{\mc Y_E})$ satisfies a local-global 
principle for rational points.  This hypothesis is equivalent to the one stated in Theorem~\ref{LGPdvr} in important special cases, e.g.\ when $K$ is  complete, or when $k$ is algebraically closed and
$G$ is stably rational.  The former case follows from \cite[Theorem~8.10(ii)]{HHK:H1}.  For the latter case, the henselization $F_P\h$ of $F$ at a point $P$ on the closed fiber is algebraic over $F$, and so every  discrete valuation on $F_P\h$ restricts to a discrete valuation on $F$.  Thus 
if $Z(F_v) \ne \emptyset$ for every $v \in \Omega_F$, then $Z((F_P\h)_v) \ne \emptyset$ for every $v \in \Omega_{F_P\h}$.  By \cite[Cor.~7.7]{BKG}, it follows that $Z(F_P\h) \ne \emptyset$ and hence $Z(F_P)\ne \emptyset$ as required. 
\end{remark}

\begin{cor}
Let $T$, $F$, $\XX$, and $G$ be as in Theorem~\ref{LGPdvr}. Assume that the linear algebraic group $G$ is retract rational over $F$, and let $Z$ be a $G$-torsor over $F$.
Then $(Z,\Omega_F)$ satisfies a local-global principle for zero-cycles of degree one.
\end{cor} 
\begin{proof}
By Remark~\ref{weaker hyp remark}, it suffices to check that $(Z_E,\Omega_{\YY_E})$ satisfies a local-global principle for rational points, for all blow-ups $\YY$ of a regular model $\XX$ and all finite separable field extensions $E/F$. That condition holds by \cite{Kra} (see also Corollary~\ref{LGP ex}).
\end{proof}

\subsection{Local-global principles with respect to discrete valuations: case of an algebraically closed residue field} \label{lgp dvr alg cl subsec}

The remainder of this section is devoted to results about torsors and 
projective homogeneous spaces over certain 2-dimensional fields over an 
algebraically closed field $k$ of characteristic zero, including 
semi-global fields.  More precisely:

\begin{hyp} \label{2dim field hyp}
Let $k$ be an algebraically closed field of characteristic zero. Let $F$ 
be one of the following:
\renewcommand{\theenumi}{\alph{enumi}}
\begin{enumerate}
\item\label{local} The fraction field of a normal, 2-dimensional, 
excellent henselian local ring with residue field $k$.
\item\label{semi-global} A one-variable function field over the fraction 
field of an excellent henselian discrete valuation ring with residue 
field $k$.
\end{enumerate}
\end{hyp}

Note that if $F$ is a field as in Hypothesis~\ref{2dim field hyp}, then finite field extensions of $F$ are also of the same type.

\begin{prop}\label{combined}
Let $F$ be a field as in Hypothesis~\ref{2dim field hyp}.  Then $F$ has 
the following properties:

\begin{enumerate}
\item\label{cd} $\cd(F) \leq 2$.

\item\label{perind} For central simple algebras over finite field extensions of  $F$,
period and index coincide.

\item\label{H1} For any semisimple simply connected group $G$ over $F$,
$H^1(F,G)=1$.

\item\label{H2}For any quasi-trivial torus $P$ over $F$, the diagonal map
$$ H^2(F,P) \to \prod_v H^2(F_v,P),$$
is injective. (Here $v$ runs through all discrete valuations on $F$.)
\end{enumerate}
\end{prop}

\begin{proof}
We first assume that $F$ is as in~Hypothesis~\ref{2dim field hyp}(\ref{local}). For the first three properties, see \cite[Thm. 1.4]{CTGP}, whereas property~(\ref{H2}) is
an immediate consequence of  \cite[Cor. 1.10]{COP}. 

Next assume $F$ is as in~Hypothesis~\ref{2dim field hyp}(\ref{semi-global}). For property~(\ref{cd}), the fraction field $K$ of an excellent henselian discrete valuation ring with algebraically closed residue field is $C_1$ by \cite[II \S3.3(c)]{Serre:CG}; 
and hence a one-variable function field $F$ over $K$ is $C_2$ by \cite[II \S4.5 Example~(b)]{Serre:CG}.  Therefore $\cd(F) \le 2$ by \cite[end of II \S4.5]{Serre:CG}, 
giving property~(\ref{cd}).  Property~(\ref{perind}) is a special case of 
\cite[Corollary~5.6]{HHK}. Property~(\ref{H2}) follows from \cite[Corollary~1.10(b)]{COP}. It remains to show property~(\ref{H1}). The absolute Galois group of the fraction field $K$ of an excellent henselian discrete valuation ring with residue field $k$ is pro-cyclic and in particular abelian, because the finite extensions of $K$ are all obtained by taking roots of the uniformizer. Hence $\bar{K}\otimes_KF/F$ is a pro-cyclic field extension, of cohomological dimension at most~$1$ by Tsen's theorem. But $F^{\operatorname{ab}}$ contains 
$\bar{K}\otimes_KF$ since the absolute Galois group of $K$ is abelian, 
and moreover this field extension is algebraic.  Hence 
$\cd(F^{\operatorname{ab}})\leq 1$ by \cite[II \S4.1 
Proposition~10]{Serre:CG}.
The statement now follows from \cite[Thm.~1.2~(v)]{CTGP} (see also \cite[end of \S 6]{PICM}).
\end{proof}

\begin{remark}
As pointed out by J. Starr, Prop.~\ref{combined}(\ref{cd})-(\ref{H1}) for fields of type~(\ref{semi-global}) in Hypothesis~\ref{2dim field hyp} can also be deduced by a localization process from global results established using the theory of rationally simply connected varieties (see \cite[Prop.~4.4]{Starr}). 
\end{remark}

Let $G$ be a connected reductive linear algebraic group over a field $K$ of characteristic zero. 
By \cite[Lemma 1.4.1]{BK} and \cite[Prop. 4.1, Cor. 5.3]{CT}, there 
exists a central extension 
\[1 \to P \to H \to G \to 1 \eqno{(*)}\] 
such that $H$ is a connected reductive group, its derived group   
$H^{ss}$ is a (semisimple) simply connected group, the quotient $H/H^{ss}$ is a torus $Q$
 and the kernel $P$ is a quasi-trivial torus. If moreover $G$ is rational over $K$, 
there exists such a presentation of $G$ for which the torus $Q$ is a quasi-trivial torus.

\begin{prop} \label{cd2} Let $F$ be any  field of characteristic zero and let $G$ be a connected reductive linear algebraic 
group over $F$. Suppose that  $H^1(F, M) = 1$ for all semisimple simply connected groups $M$ over $F$. Let $Z$ be a $G$-torsor over $F$.
\renewcommand{\theenumi}{\alph{enumi}}
\renewcommand{\labelenumi}{(\alph{enumi})}
\begin{enumerate}

\item\label{RP} If $Z$ has  a zero-cycle of degree one, then $Z$ has a rational point. 

\item\label{RP-crit}  If $G$ is rational, then $Z$ has a rational point if and only if
the  image of the class of $Z$ under the boundary map $H^1(F,G) \to H^2(F,P)$ vanishes. 
(Here $P$ is as in the above short exact sequence.) 
\end{enumerate}
\end{prop}

\begin{proof} 
For the proof of the first statement, let $1 \to P \to H \to G \to 1$ be as in $(*)$ above. Since $P$ is a quasi-trivial torus, $H^1(F, P) $ is trivial and 
 hence  we have an  exact sequence of pointed sets
 $$ 1 \to H^1(F, H) \to H^1(F, G) \to H^2(F, P).$$

 Let $[Z] \in H^1(F, G)$ denote the class of $Z$, and let $\zeta $ be the image of $[Z]$ in $H^2(F, P)$. Suppose that $Z$ has  a zero-cycle of degree one. Then since $P$ is a torus, 
 using restriction-corestriction, it follows that $\zeta$ is the trivial element. 
 Hence there exists an $H$-torsor $\widetilde{Z}$ such that its class $[\widetilde{Z}] \in H^1(F, H)$ maps to $[Z]$ in $H^1(F, G)$.
 Since for any finite field extension $L/F$, the map $H^1(L, H) \to H^1(L, G)$ has trivial kernel  and $Z$ has a zero-cycle of degree one, 
 $\widetilde{Z}$ has a zero-cycle of degree one. 
  
 Consider the exact sequence $ 1 \to H^{ss} \to H \to Q \to 1$. 
 Since $H^{ss}$ is semisimple simply connected, by the hypothesis $H^1(F, H^{ss})$ is trivial.
 Hence the map $H^1(F, H) \to H^1(F, Q)$  has trivial kernel. 
 
 Let $\bar{Z}$ be a torsor representing the image of $[\widetilde{Z}]$ in $H^1(F, Q)$. 
 Since $\widetilde{Z}$ has a zero-cycle of degree one, 
 $\bar{Z}$ has a zero-cycle of degree one. Since $Q$ is a torus,  
 once again the restriction-corestriction argument gives that 
$\bar{Z}$ has a rational point. Since  the map $H^1(F, H) \to H^1(F, Q)$  has trivial kernel, 
$\widetilde{Z}$ is trivial.  
 In particular $Z$ is trivial and hence has a rational point. 
 
 To see the second assertion, assume that $G$ is rational over $F$. Then by the discussion preceding Proposition~\ref{cd2}, we may assume that in the exact sequence
 $$1 \to H^{ss} \to H \to Q \to 1,$$
 with $H^{ss}$ semisimple simply connected, $Q$ is a quasi-trivial torus. We thus have
 $H^1(F,H)=1$. Cohomology of the exact sequence
 $$ 1 \to P \to H \to G \to 1$$
 then gives the result.
 \end{proof} 
 
 \begin{thm} \label{cd2tors}
Let 
$F$ be a field as in Hypothesis~\ref{2dim field hyp}. Let $G$ be a connected and rational  linear 
algebraic group over $F$.
Let $Z$ be a $G$-torsor over $F$. 
If $Z$ has a zero-cycle of degree one over $F_v$ for all discrete valuations $v$ of $F$, 
then $Z$ has a  rational point   over $F$.
\end{thm}
 
\begin{proof} Since $F$ has characteristic zero, the kernel $H^1(F,R_u(G))$ of the map $H^1(F,G)\rightarrow H^1(F,G/R_u(G))$ is trivial by \cite[III \S2.1, Prop.~6]{Serre:CG}; here $R_u(G)$ denotes the unipotent radical of $G$. The analogous statement holds for all field extensions $E/F$. Hence we may assume without loss of generality that $G$ is reductive. The hypotheses of Proposition~\ref{cd2} then hold, by 
Proposition~\ref{combined}(\ref{H1}). 

Let
 $$1\rightarrow P\rightarrow H\rightarrow G\rightarrow 1$$ be a central extension as in the discussion preceding Proposition~\ref{cd2}. As in the proof of that proposition, this gives rise to an exact sequence of pointed sets
 $$1\rightarrow H^1(F,H)\rightarrow H^1(F,G)\rightarrow H^2(F,P), $$
and similarly, 
 $$1\rightarrow H^1(F_v,H)\rightarrow H^1(F_v,G)\rightarrow H^2(F_v,P)$$
 for each discrete valuation $v$ of $F$. By assumption, $Z_{F_v}$ has a zero-cycle of degree one for each such $v$. By restriction-corestriction (as in the proof of Proposition~\ref{cd2}), the class of $Z_{F_v}$ maps to the trivial element in $H^2(F_v,P)$, for each $v$. 
 According to Proposition~\ref{combined}(\ref{H2}), this implies that the class of $Z$ maps to the trivial element in $H^2(F,P)$. Hence
Proposition~\ref{cd2}(\ref{RP-crit}) applies and implies the result.
 \end{proof}

\begin{remark}
For an alternative proof after reducing to the reductive 
case, first note that $Z(F_v) \ne \emptyset$ by 
Proposition~\ref{cd2}(\ref{RP}). The local case (when $F$ is as in~Hypothesis~\ref{2dim field hyp}(\ref{local})) now follows from \cite[Corollary~7.7]{BKG}.  In the case of a function field of a normal curve $\XX$ over an excellent henselian discrete valuation ring $T$ as in~Hypothesis~\ref{2dim field hyp}(\ref{semi-global}), the local case implies that $Z$ has a point over the fraction field of the henselization of the local ring of $\XX$ at any closed point $P$.  Hence $Z$ also has a point over the fraction field of the complete local ring at $P$. It also has a point over $F_\eta$, for each generic point $\eta$ of the closed fiber.  Thus it has a point over the function field of $\XX \times_T \wh T$ by \cite[Theorem~5.10]{HHK:H1}. This case of the theorem then follows from Lemma~\ref{henselian pt la} above.
\end{remark}
  
Below we study local-global principles for zero-cycles on projective homogeneous spaces. We use a criterion for the existence of rational points from~\cite{CTGP}. We begin by recalling some notation and facts (loc. cit., p. 333--335, which closely follows work of Borovoi \cite{B}).

Let $F$ be a field of characteristic zero. Let $G$ be a connected reductive linear algebraic 
group over $F$ and let $Z$  be a projective homogeneous $G$-space. 
Let $\bar{H}$ be the isotropy group of an $\bar{F}$-point of $Z$; note that since $Z$ is projective, $\bar{H}$ is parabolic and hence connected. As in \cite{CTGP}, one can define an associated $F$-torus
 $H^{\tor}$ (this is an $F$-form of the maximal torus quotient of $\bar{H}$).
 Because $Z$ is projective, the $F$-torus $H^{\tor}$ is a quasi-trivial torus by
 \cite[Lemma 5.6]{CTGP}. 
 
As in \cite{CTGP}, one may further define an $F$-kernel $L= (\bar{H},\kappa)$, a cohomology set $H^2(F, L)$, and a class $\eta(Z) \in H^2(F, L)$ associated to $Z$.
This class is a {\em neutral class} if and only if $Z$ comes from 
a $G$-torsor; i.e., if and only if there exists a $G$-torsor $Y$ and a $G$-equivariant morphism $Y\rightarrow Z$.
There is a natural map $t_* : H^2(F, L) \to H^2(F, H^{\tor})$
which is functorial in the field $F$  (loc. cit.), and which sends neutral classes in $H^2(F, L) $ to
the trivial element in $H^2(F,H^{\tor})$. 

The following proposition is an immediate consequence of \cite[Prop. 5.4]{CTGP} and will be the key ingredient in the proof of the theorem below.

\begin{prop}\label{RPcrit}
Suppose that $F$ is a field which satisfies properties~(\ref{cd}),~(\ref{perind}), and~(\ref{H1}) in Proposition~\ref{combined}. Let $G$ be a semisimple simply connected linear algebraic group over $F$, and let $Z$ be a projective homogeneous space under $G$. Then $Z(F) \neq \emptyset$ if and only if
$t_*(\eta(Z))=1 \in  H^2(F, H^{\tor})$. 
\end{prop}

\begin{proof}
If $Z(F)$ contains a rational point $x$ then the map $G\rightarrow Z$ given by the action of $G$ on $x$ shows that $\eta(Z)$ is neutral, and thus $t_*(\eta(Z))=1 \in  H^2(F, H^{\tor})$. Conversely, if $t_*(\eta(Z))=1$, then $\eta(Z)$ is neutral by \cite[Prop. 5.4]{CTGP}; i.e., it comes from a $G$-torsor. By assumption, $H^1(F,G)$ consists of a single element, hence that torsor has an $F$-rational point which maps to a point on $Z$, showing $Z(F)\neq \emptyset$.
\end{proof}

\begin{thm} \label{cd2proj} 
Let $F$ be a field of characteristic zero which satisfies properties~(\ref{cd}),~(\ref{perind}),~(\ref{H1}) in Proposition~\ref{combined}. 
Let  $G$ be a connected linear algebraic 
group over $F$, and let $Z$  be a projective homogeneous $G$-space. 
If $Z$ has  a zero-cycle of degree one, then $Z$ has a rational point. In particular, this assertion holds whenever $F$ is a field as in Hypothesis~\ref{2dim field hyp}.
\end{thm}

\begin{proof} 
Let $R(G)$ be the radical of $G$, i.e., the maximal connected solvable subgroup of $G$. This is contained in any parabolic subgroup of $G$.
Since $Z$ is projective, the action of $G$ on $Z$ factors through $G^{ss} = G/R(G)$.
Let  $G^{sc} \to G^{ss}$ be the simply connected cover of 
$G^{ss}$. Thus we may view $Z$ as a projective homogenous $G^{sc}$-space. 
Replacing $G$ by $G^{sc}$, we may therefore assume that $G$ is semisimple and simply connected 
(cf. proof of \cite[Corollary 5.7]{CTGP}). 

If  $Z$ has a zero-cycle of degree one,
  there exist finite field extensions $F_i/F$ such that $Z(F_i) \neq \emptyset$ for all $i$
and the g.c.d. of the degrees of $F_i/F$ is~$1$.  
Since $Z(F_i) \neq \emptyset$, the image of $t_*(\eta(Z))$ in $H^2(F_i, H^{\tor})$ is trivial by Proposition~\ref{RPcrit}. 
Using a restriction-corestriction argument, we conclude that $t_*(\eta(Z))$ is trivial 
in $H^2(F, H^{\tor})$. One then concludes that $Z(F)\neq \emptyset$ by a second application of Proposition~\ref{RPcrit}. The last statement follows from Proposition~\ref{combined}.
\end{proof}

If $F$ is the function field of a $p$-adic curve, there are projective homogeneous spaces under simply connected groups which admit zero-cycles of degree one, but with no rational points. See \cite{Parimala:Asian}.

\begin{thm}\label{LGPproj}
Let $F$ be a field as in Hypothesis~\ref{2dim field hyp}, let
$G$ be  a connected linear algebraic group over $F$, and
let $Z$ be a projective homogeneous $G$-space. 
 If   $Z$ has a zero-cycle of degree one over 
$F_v$ for each discrete valuation $v$ of $F$,  then $Z$ has a  rational point over $F$.
\end{thm}

\begin{proof} 
As in the first paragraph of the proof of Theorem~\ref{cd2proj}, we may reduce to the case that $G$ is semisimple simply connected. If $Z$ has a zero-cycle of degree one over each $F_v$,
then a corestriction-restriction argument shows that the class
$t^*(\eta(Z)) =1 \in H^2(F, H^{\tor})$  
has trivial image in each $H^2(F_v,H^{\tor})$.
Since $H^{\tor}$ is  a quasi-trivial torus, Proposition~\ref{combined}(\ref{H2}) implies that
$t^*(\eta(Z)) =1 \in H^2(F, H^{\tor})$.
An application of Proposition~\ref{RPcrit} then yields $Z(F) \neq \emptyset$.
\end{proof}

\bigskip

\noindent{\bf Author Information:}\\

\noindent Jean-Louis Colliot-Th\'el\`ene\\
Laboratoire de Math\'ematiques d'Orsay, Univ.\ Paris-Sud, CNRS, Universit\'e Paris-Saclay, 91405 Orsay, France\\
email: jlct@math.u-psud.fr

\medskip

\noindent David Harbater\\
Department of Mathematics, University of Pennsylvania, Philadelphia, PA 19104-6395, USA\\
email: harbater@math.upenn.edu

\medskip

\noindent Julia Hartmann\\
Department of Mathematics, University of Pennsylvania, Philadelphia, PA 19104-6395, USA\\
email: hartmann@math.upenn.edu

\medskip

\noindent Daniel Krashen\\
Department of Mathematics, University of Georgia, Athens, GA 30602, USA\\
email: dkrashen@math.uga.edu

\medskip

\noindent R.~Parimala\\
Department of Mathematics and Computer Science, Emory University, Atlanta, GA 30322, USA\\
email: parimala@mathcs.emory.edu

\medskip

\noindent V.~Suresh\\
Department of Mathematics and Computer Science, Emory University, Atlanta, GA 30322, USA\\
email: suresh@mathcs.emory.edu

\medskip

\noindent The authors were supported on NSF collaborative FRG grant: DMS-1463733 (DH and JH), DMS-1463901 (DK), DMS-1463882 (RP and VS).  Additional support was provided by NSF collaborative FRG grant DMS-1265290 (DH); NSF RTG grant DMS-1344994 (DK); NSF DMS-1401319 (RP); NSF DMS-1301785 (VS); and a Simons Fellowship (JH).

\end{document}